\documentclass[11pt,reqno]{amsart}

\usepackage[foot]{amsaddr} 

\usepackage{amssymb,amsthm}
\usepackage{latexsym}
\usepackage{color}
\usepackage{xcolor}
\usepackage{mathrsfs}
\usepackage{enumitem}
\usepackage[margin=1in]{geometry}
\usepackage{cite}

\usepackage{graphicx}

\usepackage{hyperref}
\hypersetup{
    colorlinks,
	linkcolor={black!30!blue},
	citecolor={black!30!blue},
	urlcolor={black!30!blue},
	breaklinks={true},
	pdftitle = {Absence of Bound States for Quantum Walks and CMV Matrices via Reflections},
	pdfauthor = {Christopher Cedzich,
		Jake Fillman},
	pdfsubject = { },
	pdfkeywords = {unitary almost Mathieu operator, 
		singular continuous spectrum, 
		resonances,
		quantum walks, 
		CMV matrices, 
		delocalization, 
		quasi-periodicity}
}

\usepackage{mathtext}
\usepackage[T1,T2A]{fontenc}
\DeclareSymbolFont{T2Aletters}{T2A}{cmr}{m}{it}
\usepackage[utf8x]{inputenc}

\newcommand\Ni{\CYRI}

\date{\today}

\newcommand{\bbC}{{\mathbb{C}}}
\newcommand{\bbD}{{\mathbb{D}}}
\newcommand{\bbN}{{\mathbb{N}}}

\newcommand{\bbS}{{\mathbb{S}}}
\newcommand{\bbT}{{\mathbb{T}}}
\newcommand{\bbZ}{{\mathbb{Z}}}

\newcommand{\calA}{{\mathcal{A}}}
\newcommand{\calE}{{\mathcal{E}}}
\newcommand{\calL}{{\mathcal{L}}}
\newcommand{\calM}{{\mathcal{M}}}
\newcommand{\calR}{{\mathcal{R}}}

\newcommand{\scrL}{{\mathscr{L}}}

\newcommand{\vertiii}[1]{{\left\vert\kern-0.25ex\left\vert\kern-0.25ex\left\vert #1 
    \right\vert\kern-0.25ex\right\vert\kern-0.25ex\right\vert}}

\newcommand{\subst}{{\mathsf{S}}}
\newcommand{\TM}{{\mathrm{TM}}}
\DeclareMathOperator{\orb}{orb}
\DeclareMathOperator{\hull}{hull}

\DeclareMathOperator{\shift}{sh}

\newcommand{\refl}{{\mathrm{Refl}}}
\newtheorem{theorem}{Theorem}[section]
\newtheorem{lemma}[theorem]{Lemma}
\newtheorem{coro}[theorem]{Corollary}

\theoremstyle{definition}

\newtheorem{definition}[theorem]{Definition}

\newtheorem{remark}[theorem]{Remark}
\newtheorem*{notation}{Notation}

\allowdisplaybreaks \numberwithin{equation}{section}

\makeatletter
\def\subsection{\@startsection{subsection}{2}%
	\z@{.5\linespacing\@plus.7\linespacing}{.5\linespacing}%
	{\normalfont\scshape\centering}}
\makeatother

\definecolor{purple}{rgb}{.5,0,1}

\title[Continuous Spectrum for CMV matrices via Reflections]{Absence of Bound States for Quantum Walks\\ and CMV Matrices via Reflections}
\author[C.\ Cedzich]{Christopher Cedzich}
\email{\href{mailto:cedzich@hhu.de}{cedzich@hhu.de}}
\address{Heinrich Heine Universit\"at D\"usseldorf, Universit\"atsstr. 1, 40225 D\"usseldorf, Germany}
\author[J.\ Fillman]{Jake Fillman}
\email{\href{mailto:fillman@tamu.edu}{fillman@tamu.edu}}
\address{Department of Mathematics, Texas A\&M University, College Station, TX 77843-3368, USA}

\date{}

\begin{document}
\maketitle

\begin{abstract}
We give a criterion based on reflection symmetries in the spirit of Jitomirskaya--Simon to show absence of point spectrum for (split-step) quantum walks and Cantero--Moral--Vel\'azquez (CMV) matrices. 
To accomplish this, we use some ideas from a recent paper by the authors and collaborators to implement suitable reflection symmetries for such operators. 
We give several applications. 
For instance, we deduce arithmetic delocalization in the phase for the unitary almost-Mathieu operator and singular continuous spectrum for generic CMV matrices generated by the Thue--Morse subshift.
\end{abstract}
\medskip

\textbf{Keywords.}  unitary almost Mathieu operator, 
		singular continuous spectrum, 
		resonances,
		quantum walks, 
		CMV matrices, 
		delocalization, 
		quasi-periodicity
		\medskip
		
		\textbf{MSC 2020 Classifications. Primary:} 47B93   \textbf{Secondary: } 47A10, 47B36, 58J51, 81Q10
		\bigskip

\section{Introduction}

One of the fundamental issues in the theory of ergodic operators is to establish the presence of Anderson localization (i.e., pure point spectrum with exponentially decaying eigenfunctions) whenever one suspects it to be present, and furthermore, one often suspects Anderson localization occurs for operator families in the regime of positive Lyapunov exponents. Indeed, positivity of the Lyapunov exponent is often the first step on a journey towards localization and provides a key input: namely, (via Ruelle's theorem \cite{Ruelle1979PMIHES}) positivity of the Lyapunov exponent ensures that for any energy for which the pointwise Lyapunov exponent exists, one has an exponential dichotomy of solutions and therefore generalized eigenvalues must decay exponentially. However, the phenomenon in question is in general very delicate, since one must be careful of the exceptional set on which the pointwise Lyapunov exponent fails to exist. This can be overcome, and hence the desired localization can be proved, using a variety of techniques, such as multi-scale analysis, fractional moment analysis, large deviations, repulsion of singular clusters, elimination of double resonances, and others; for references and further reading, we direct the reader to \cite{aizenmanRandomOperatorsDisorder2015, kirschInvitationRandomSchroedinger2007, ESO1,ESO2}.
These approaches have been pursued and implemented in a number of works concerning quasi-periodic \cite{CFO2023CMP, GuoPiao2020LAA, WangDamanik2019JFA, yang2022localization}, almost-periodic \cite{CFLOZ}, skew-shift \cite{CedzichWerner2021CMP, Kruger2013IMRN, LPG2023JFA} and random \cite{ASW2011JMP, BDFGVWZ2017TAMS, zhu2021localization, joyeDynamicalLocalizationQuantum2010a,hamzaLocalizationRandomUnitary2006} CMV matrices. 

One is also naturally interested in understanding structures that preclude localization in the regime of positive exponents (either globally, or locally in energy). The present paper accomplishes precisely this task for CMV operators in the presence of suitable reflection symmetries by synthesizing ideas of \cite{CFLOZ} to exploit gauge-invariance in order to implement the approach of Jitomirskaya--Simon \cite{jitomirskayaOperatorsSingularContinuous1994}. The main results of this work in particular show that several localization results recently obtained for the unitary almost-Mathieu operator \cite{CFO2023CMP, yang2022localization} and mosaic model \cite{CFLOZ} are sharp in the sense that localization for almost every phase cannot be strengthened to localization for all phases. Furthermore, the method gives generic absence of eigenvalues for CMV matrices and quantum walks generated by subshifts containing many palindromes.  In particular, we obtain the first results on singular continuous spectrum for CMV matrices generated by the Thue--Morse subshift, to our knowledge.

 Let us briefly recall some of the background literature. In the setting of self-adjoint operators Jitomirskaya--Simon gave a criterion for Schr\"odinger operators to have purely continuous spectrum \cite{jitomirskayaOperatorsSingularContinuous1994}. This was adapted to palindromic subshift operators by Hof--Knill--Simon \cite{HofKniSim1995CMP} and developed further in works such as Koslover \cite{kosloverJacobiOperatorsSingular2005} and Jitomirskaya--Liu \cite{JitoLiu2023JEMS}. 
As was noted in \cite{CFLOZ}, arguments based on reflections to be more delicate (even outright impossible) for standard CMV operators, due to a lack of appropriate reflection symmetries. However, the authors and collaborators observed in \cite{CFLOZ} that one can overcome this by passing to the setting of so-called \emph{generalized} CMV matrices via complexification and relating the complexified CMV matrix to the initial operator via a diagonal gauge transformation that \emph{preserves} the Verblunsky coefficients. Let us also mention a different obstruction to localization, known as the Gordon criterion \cite{Gordon1976}, is based on the presence of suitable repetitions, rather than reflections. This is significantly more straightforward and has already been worked out for CMV matrices in \cite{F2017PAMS}; see also \cite{LDZ2022TAMS} for an arithmetic strengthening and \cite{Ong2012JMAA,Ong2014JSP} for earlier results. Recently, Avila and Damanik have announced another approach to delocalization that applies to ergodic self-adjoint operators in arbitrary dimensions. \cite{AvilaDamanikDetermDeloc}

CMV matrices have physical significance, as they are (equivalent to) one-dimensional split-step quantum walks, which was first observed in \cite{canteroMatrixvaluedSzegoPolynomials2010} and elaborated upon in \cite{CFLOZ}. Split-step quantum walks are the simplest type of the rich class of quantum walks which have gained popularity in recent years as models of discrete time quantum dynamics. In this sense, quantum walks play a role similar to that of Jacobi matrices for continuous time evolutions generated by general self-adjoint operators, and one is therefore naturally interested in their dynamical \cite{ASW2011JMP, SpacetimeRandom,grimmett2004weak, AVWW2011JMP,joyeDynamicalLocalizationQuantum2010a, aschLowerBoundsLocalisation2019, joyeDynamicalLocalizationDdimensional2012} and spectral \cite{ewalks, CFGW2020LMP, richardQuantumWalksAnisotropic2018, bourgetSpectralAnalysisUnitary2003, joyeDensityStatesThouless2004} properties.
\medskip

 We formulate our main results precisely in Section~\ref{sec:results} and prove them in Section~\ref{sec:proofs}.

\subsection*{Acknowledgements}
C.\ C.\ was supported in part by the Deutsche Forschungsgemeinschaft (DFG, German Research Foundation) under the grant number 441423094.
J.\ F.\ was supported in part by National Science Foundation (NSF) grant DMS-2213196.
J.\ F.\ thanks the American Institute of Mathematics for hospitality during a recent SQuaRE program, which facilitated the work.
The authors are grateful to the reviewers for carefully reading the manuscript and for helpful comments.

\section{Results}\label{sec:results}
\subsection{General Results}
We begin by recalling the notion of generalized extended CMV (GECMV) matrices from \cite{CFLOZ}. Let $\bbS^3 = \{(\alpha,\rho) \in \bbC^2 : |\alpha|^2 + |\rho|^2 =1\}$ and consider a sequence $\{(\alpha_n,\rho_n)\} \in (\bbS^3)^\bbZ$. In the standard basis of $\ell^2(\bbZ)$, the corresponding GECMV matrix takes the form
\begin{equation} \label{eq:gecmv}
\calE
=
\calE_{\alpha,\rho}
=
\begin{bmatrix}
\ddots & \ddots & \ddots & \ddots &&&&  \\
& \overline{\alpha_0\rho_{-1}} & \boxed{-\overline{\alpha_0}\alpha_{-1}} & \overline{\alpha_1}\rho_0 & \rho_1\rho_0 &&&  \\
& \overline{\rho_0\rho_{-1}} & -\overline{\rho_0}\alpha_{-1} & {-\overline{\alpha_1}\alpha_0} & -\rho_1 \alpha_0 &&&  \\
&&  & \overline{\alpha_2\rho_1} & -\overline{\alpha_2}\alpha_1 & \overline{\alpha_3} \rho_2 & \rho_3\rho_2 & \\
&& & \overline{\rho_2\rho_1} & -\overline{\rho_2}\alpha_1 & -\overline{\alpha_3}\alpha_2 & -\rho_3\alpha_2 &    \\
&& && \ddots & \ddots & \ddots & \ddots &
\end{bmatrix},
\end{equation}
where we boxed the $(0,0)$ matrix element.  A \emph{standard} CMV matrix has the form \eqref{eq:gecmv} with $\rho_n \geq 0$ for all $n$; equivalently, $\rho_n = \sqrt{1-|\alpha_n|^2}$. The standard CMV matrix has an important role in mathematical physics, spectral theory, and orthogonal polynomials; for a textbook discussion that includes many important cases (e.g.\ periodic, random, and almost-periodic), see \cite{Simon2005OPUC1,Simon2005OPUC2}. For any sequence $\alpha = \{\alpha_n\}\in \bbD^\bbZ$, the associated standard CMV matrix and its cousin with complexified $\rho$ terms will be denoted by
\[\calE_{\alpha}:= \calE_{\alpha,\sqrt{1-|\alpha|^2}}, \quad \calE_\alpha^\bbC := \calE_{\alpha,i\sqrt{1-|\alpha|^2}}.\]
 By \cite{CFLOZ}, $\calE_\alpha$ and $\calE_\alpha^\bbC$ are equivalent via a gauge transform.

Given $u \in \bbC^\bbZ$, and $\zeta \in \frac{1}{2} \bbZ$, the \emph{reflection} of $u$ through the center $\zeta$ is given by
\begin{equation}
	(\calR_\zeta u)_n = u_{2\zeta-n}.
\end{equation}
In order to get useful estimates, we will need Verblunksy pairs with exponentially good approximate reflection symmetries locally on intervals that are exponentially long (with respect to the center of reflection). To generate such symmetries for a suitably rich set of phases in the ergodic setting below, it is useful to begin with a sequence that obeys exact symmetries globally. The following notions make this precise.

We say that $\alpha = \{\alpha_n\} \in \bbD^\bbZ$ is \emph{reflection symmetric} about the center $\zeta$ if 
\begin{equation} \label{eq:alphasymm}
	(\calR_\zeta \alpha)_{n} \equiv \overline{\alpha_n}.
\end{equation}

 It is no loss of generality to assume that the center of reflection is $0$, which we will do throughout the paper: we simply say that $\alpha$ is \emph{reflection symmetric} if it is reflection symmetric about $\zeta=0$.

 The following definition captures precisely the approximate local reflection symmetries that we will need.

\begin{definition} \label{def:JitoSim:Brefl}
	Given $B>0$ and $\zeta \in \tfrac{1}{2}\bbZ$, we say that $\alpha \in \bbD^\bbZ$ is $(B,\zeta)$\emph{-reflective} if
	\begin{equation}
		\big| (\calR_\zeta \alpha)_n -\overline{\alpha_n} \big| < e^{-B|\zeta|} \quad \forall n \in \bbZ \text{ such that } |\zeta-n|\leq e^{B|\zeta|}.
	\end{equation}
	We say that $\alpha$ is $B$\emph{-reflective} if it is $(B,\zeta)$-reflective for infinitely many $\zeta \in \frac{1}{2}\bbZ$. Clearly, $\alpha$ is $(B,\zeta)$-reflective for every $B>0$ whenever it is reflection symmetric about $\zeta$, but not conversely.
\end{definition}

Throughout, we assume that the Verblunsky coefficients are bounded away from $1$ in modulus:
\begin{equation} \label{eq:VCbd}
\|\alpha\|_\infty = r_0 < 1,
\end{equation}
which in particular implies 
\begin{equation} \label{eq:rhobound}
|\rho_n| \geq\big(1-r_0^2\big)^{1/2} =: c_0>0
\end{equation}
 uniformly in $n$. 

The main result is the following:

\begin{theorem} \label{t:JitomirskayaSimon:refl}
Assume $\alpha$ satisfies \eqref{eq:VCbd}.
There is a constant $B_0$ such that the following is true: 
If $\alpha$ is $B$-reflective with $B > B_0$,	then $\calE=\calE_{\alpha,\rho}$ has empty point spectrum for any admissible\footnote{That is, any choice of $\rho$ satisfying $|\alpha_n|^2 + |\rho_n|^2 \equiv 1$.} choice of $\rho$; in particular, the standard CMV matrix $\calE_\alpha$ has empty point spectrum.
\end{theorem}

\begin{remark} \label{rem:mainThmComments}
 A few comments are in order.
\begin{enumerate}[label={\rm(\alph*)}]
\item A straightforward attempt to imitate the proof of \cite{jitomirskayaOperatorsSingularContinuous1994} for standard CMV matrices leads to roadblocks arising from the lack of relevant symmetries. For instance, it is clear that an estimate of the form \eqref{eq:rhosymm} is needed to obtain useful estimates on the quantities in, for example, \eqref{eq:longWronskiCalc} and \eqref{lem:25:Phiipm0setup}. Thus, in addition to the identity \eqref{eq:alphasymm} for the $\alpha$'s, one also must have a symmetry of the type \eqref{eq:rhosymm} for the $\rho$'s, and such a symmetry cannot hold for standard CMV matrices, so one absolutely must complexify to reveal the necessary symmetry. 
\item In fact, as the reader will see, even with the new ideas using gauge transforms to relate to the complexified $\rho$ case, the computations are somewhat more laborious in the present setting. For instance, one may compare the proof of Lemma~\ref{lem:W-W} to the corresponding computation in \cite[Step~1]{jitomirskayaOperatorsSingularContinuous1994}. Similarly, the absence of symplectic symmetry (even after complexification) makes the comparison of ``reflected'' and inverted transfer matrices quite delicate
\item If $\alpha$ satisfies \eqref{eq:VCbd}, then the hypothesis of Theorem~\ref{t:JitomirskayaSimon:refl} is fulfilled whenever $\alpha$ is even and real-valued.
\item The largeness condition on $B$ is concrete. To ensure that $z \in \partial \bbD$ is not an eigenvalue of $\calE_\alpha$, one simply needs $B$ to be large enough that
\begin{equation}
\lim_{n \to \infty} e^{-Bn}\big\|N_{\ell+n-1,z}\cdots N_{\ell+1,z}N_{\ell,z}\big\| = 0
\end{equation}
uniformly in $\ell \in \bbZ$, where $N_{n,z}$ denote the transfer matrices (see Lemma~\ref{lem:B}).
Thus, for ergodic models over strictly ergodic base dynamics (for which one may apply Furman's uniform subadditive ergodic theorem \cite{Furman1997AIHP}), one can phrase the relevant bound in terms of the Lyapunov exponent \emph{locally} in the spectral parameter and deduce \emph{arithmetic} delocalization; compare Theorem~\ref{t:JL:quant} and Corollary~\ref{coro:genQP}.

\item When applying the results to general quantum walks that are not given by GECMV matrices, some care is needed. For instance, the electric walks in \cite{ewalks,CedzichWerner2021CMP} have symmetric coins, but application of the gauge transform leads to a CMV matrix that does not satisfy the necessary symmetries, so our results do not apply in that setting.

\item It follows readily from the result of Damanik--Killip \cite{DamanikKillip2000JFA} that for any almost-periodic sequence $\alpha$ that is not limit-periodic, the set of reflective $\omega \in \hull(\alpha)$ has zero Haar measure.
\end{enumerate}
\end{remark}

Theorem~\ref{t:JitomirskayaSimon:refl} can be applied to almost-periodic CMV matrices. Recall that $\alpha \in \bbC^\bbZ$ is called \emph{almost-periodic} if its orbit
\[\orb(\alpha):= \{\alpha(\cdot - n) : n \in \bbZ\}\]
is relatively compact in $\ell^\infty(\bbZ)$. In that case, the closure (which is compact) is called the \emph{hull}:
\begin{equation} \label{eq:hulldef}
\hull(\alpha) := \overline{\orb(\alpha)}^{\|\cdot\|_\infty}.
\end{equation}

\begin{coro} \label{coro:JitomirskayaSimon:main}
	Assume $\alpha :\bbZ \to \bbD$ satisfies \eqref{eq:VCbd}. If $\alpha$ is reflection-symmetric and almost-periodic, then there is a dense $G_\delta$ subset $\Omega_0 \subseteq \hull(\alpha)$ such that the extended CMV matrix $\calE_\omega$ has empty point spectrum for all $\omega \in \Omega_0$.
\end{coro}

Note that in the setting of Corollary~\ref{coro:JitomirskayaSimon:main}, $\omega \in \hull(\alpha)$, so the Verblunsky coefficient sequence associated with $\calE_\omega$ is simply the sequence $\omega$ itself, that is, $\alpha_\omega(n)=\omega(n)$.
In addition, as alluded to in Remark~\ref{rem:mainThmComments}, one can make Corollary~\ref{coro:JitomirskayaSimon:main} quantitative and can exclude eigenvalues locally in the spectral parameter. In order to do that, we need the Lyapunov exponent $L(z)$ (which will be defined in \eqref{eq:LEdef}).
\begin{theorem} \label{t:JL:quant}
Assume $\alpha:\bbZ \to \bbD$ satisfies \eqref{eq:VCbd}. If $\alpha$ is almost-periodic and $B$-reflective, then, for any admissible $\rho$, $\calE_{\alpha,\rho}$ has no eigenvalues in the region $\{z \in \partial \bbD :  2L(z) < B\}$.
\end{theorem}
This allows us to establish a corollary giving arithmetic delocalization in the spirit of \cite[Theorem~4.2]{JitoLiu2023JEMS}. To formulate the result, recall that a topological group is called \emph{monothetic} if it contains a dense cyclic subgroup (e.g., this class of groups includes $\bbZ$, $\bbZ_p$, $\bbT^d$, $\bbZ_p \times \bbT^d$, and certain procyclic groups such as the $p$-adic integers). Given a metrizable compact monothetic $\Omega$, a generator $\beta$ of a dense cyclic subgroup, and $f \in C(\Omega,\bbC)$, the sequence $(f(n\beta+\omega))_{n \in \bbZ}$ is almost-periodic\footnote{That is, its translates are precompact in $\ell^\infty(\bbZ)$.} for any choice of $\omega \in \Omega$. This is a standard fact whose proof can be found in various places, e.g., \cite[Section~3]{DF2020LP}.

 For a metrizable compact monothetic group $\Omega$ with translation-invariant metric $\mathrm{dist}$, we denote
\[\vertiii{\omega} = \vertiii{\omega}_\Omega: = \mathrm{dist}(\omega,0), \quad \omega \in \Omega.\]
For $\beta,\omega \in \Omega$, define:
\begin{equation} \label{eq:deltaBetaOmegaDef}
\delta(\beta,\omega) = \limsup_{|n|\to\infty} \left[- \frac{\log\vertiii{2\omega+n\beta}}{|n|}\right].
\end{equation}
\begin{coro} \label{coro:genQP}
Suppose $\Omega$ is a compact metrizable monothetic group, that $\beta,\omega \in \Omega$, and that the subgroup generated by $\beta$ is dense in $\Omega$. If
 $f:\Omega \to \bbD$ is Lipschitz continuous and satisfies $f(-\omega) \equiv \overline{f(\omega)}$, then for each $\omega \in \Omega$, the CMV matrix $\calE_\omega$ with coefficients
\[\alpha_\omega(n) = f(n\beta + \omega)\]
has no eigenvalues in the region $\{z \in \partial \bbD : L(z) < \delta(\beta,\omega)\}$.
\end{coro}
Let us briefly point out that \cite[Theorem~4.2]{JitoLiu2023JEMS} is formulated for $\Omega = \bbT$, but the proof they give carries through with cosmetic changes to the more general setting described in Corollary~\ref{coro:genQP}. Moreover, we need this slight generalization since the UAMO (viewed as a dynamically defined CMV matrix) is most naturally viewed as having base dynamics on $\bbT \times \bbZ_2$, as discussed below in \eqref{eq:UAMOdynamicalDef}.

\subsection{Examples and Applications}

The criterion in the main results can be applied to models of interest. For instance, localization questions for the unitary almost-Mathieu operator (UAMO) introduced in \cite{CFO2023CMP, FOZ2017CMP} have been studied in \cite{CFO2023CMP, yang2022localization}.
In light of \cite{CFO2023CMP,CFLOZ}, the UAMO family is gauge-equivalent to the family of GECMV matrices given by
    \begin{alignat*}{3} 
		\alpha_{2n-1} &= \lambda_2 \cos(2\pi(n\Phi+\theta)), \qquad\qquad &\alpha_{2n} &= \lambda_1',
    \end{alignat*}
where we abbreviate $\lambda_1':= (1-\lambda_1^2)^{1/2}$.
Let us denote the corresponding operators by $\mathcal{E}_{\lambda_1,\lambda_2,\Phi,\theta}^{\rm UAMO}$. Notice that we use the cosine rather than the sine (which is the definition given in \cite{CFO2023CMP}) since the cosine function is even and we want to emphasize reflection symmetries. 
\begin{coro} \label{coro:uamo}
For any  $\lambda_1,\lambda_2 \in (0,1)$ and irrational $\Phi$, $\calE_{\lambda_1, \lambda_2, \Phi, \theta}^{\rm UAMO}$ has empty point spectrum whenever
\begin{equation} \label{eq:deltaBoundUAMO}
\frac{1}{2} \delta(\Phi,\theta) > \log\left[ \frac{\lambda_2(1+\lambda_1')}{\lambda_1(1+\lambda_2')} \right].
\end{equation}
In particular, $\calE_{\lambda_1, \lambda_2, \Phi, \theta}^{\rm UAMO}$ has empty point spectrum for generic $\theta \in \bbT$ and has purely singular continuous spectrum whenever \eqref{eq:deltaBoundUAMO} and $\lambda_1<\lambda_2$ hold.
\end{coro}

The same result can be applied to other almost-periodic CMV matrices, such as the mosaic model considered in \cite{CFLOZ}, which can be given by an additional choice of parameter $s \in \bbZ_+$ and defining
\begin{equation}
\alpha_{2n} = \lambda_1', \quad \alpha_{2kn-1} = \begin{cases} \lambda_2 \cos(2\pi(n\Phi+\theta)) & k \in s\bbZ \\
0 & \text{otherwise.} \end{cases}
\end{equation}
We write the resulting operator as $\calE_{s,\lambda_1, \lambda_2, \Phi, \theta}^{\rm mosaic}$.

\begin{coro} \label{coro:mosaic}
For any $s \in \bbZ$, $\lambda_1,\lambda_2 \in (0,1)$, and irrational $\Phi$, $\calE_{s,\lambda_1, \lambda_2, \Phi, \theta}^{\rm mosaic}$ has empty point spectrum for generic $\theta \in \bbT$.
\end{coro}

Finally, let us mention that the criterion in Theorem~\ref{t:JitomirskayaSimon:refl} may be applied to ergodic CMV matrices taking finitely many values, which occur naturally in the context of subshifts. Given a finite set $\calA \subseteq \bbD$ (the \emph{alphabet}), a \emph{subshift} is any compact\footnote{Naturally, $\calA$ is given the discrete topology and $\calA^\bbZ$ is in turn given the product topology.} set $\Omega \subseteq \calA^\bbZ$ that is invariant under the action of the shift map $\shift:\calA^\bbZ \to \calA^\bbZ$ given by
\[
[\shift \omega](n) = \omega(n+1).
\] 
It is well-known that a subshift $(\Omega,\shift)$ is minimal (in the sense that the shift-orbit of every $\omega \in \Omega$ is dense in $\Omega$) if and only if there is a set
\[ \scrL = \scrL(\Omega) \subseteq \bigcup_{n \in \bbN} \calA^n \]
with the property that $\scrL$ is precisely the set of finite strings occurring in every $\omega \in \Omega$, that is,
\[ \{\omega_j \omega_{j+1} \cdots \omega_{j+m-1} : j \in \bbZ, \ m \geq 1\} =\scrL\]
for every $\omega \in \Omega$.
 The set $\scrL$ is called the \emph{language} of the subshift.

\begin{coro} \label{coro:palin}
If $(\Omega,\shift)$ is a minimal subshift whose language contains arbitrarily long palindromes, then $\calE_\omega$ has empty point spectrum for generic $\omega \in \Omega$.
\end{coro}

 In particular, let us mention that Corollary~\ref{coro:palin} can be applied to the subshift generated by the Thue--Morse substitution $\subst_\TM:a \mapsto ab$, $b \mapsto ba$ over the alphabet $\calA = \{a,b\}$. The corresponding CMV matrices and quantum walks were studied from a dynamical perspective \cite{DFO2016JMPA, F2017IIIS} and were known to have purely singular spectrum by Kotani theory \cite{Simon2005OPUC2} and in fact Cantor spectrum of zero Lebesgue measure \cite{DamLen2007JAT}, but singular continuous spectrum could not be established until now.   Since this is an important example in the theory of aperiodic order, we formulate this explicitly as a corollary:

\begin{coro}
If $\Omega \subseteq \{a,b\}^\bbZ$ is the subshift generated by the Thue--Morse substitution, then for generic $\omega \in \Omega$, $\calE_\omega$ has purely singular continuous spectrum.
\end{coro}

 To the best of our knowledge, this gives the first result about singular continuity for CMV matrices or quantum walks generated by the Thue--Morse substitution. For Schr\"odinger operator versions of these results, see \cite{HofKniSim1995CMP}.

\section{Proofs of Results} \label{sec:proofs}
One studies $\calE$ and its spectral properties via the generalized eigenvalue equation
\begin{equation} \label{eq:generalizedEvalEq}
\calE u = zu, \quad z \in \bbC^*\equiv \bbC \setminus \{0\}, \ u \in \bbC^\bbZ,
\end{equation}
where we emphasize that no assumption is made about $u$ beyond being a complex-valued sequence.

For later use note that $\calE$ acts on coordinates via
\begin{align}\label{eq:E_action_2n}
    [\calE_{\alpha,\rho} u](2n)
&= \overline{\alpha_{2n}}\Big[\overline{\rho_{2n-1}}u(2n-1)  -\alpha_{2n-1} u(2n)\Big] \\
\nonumber & \qquad \qquad    + \rho_{2n}\Big[\overline{\alpha_{2n+1}} u(2n+1)+\rho_{2n+1} u(2n+2)\Big]\\[4mm]
\label{eq:E_action_2n+1}
    [\calE_{\alpha,\rho} u](2n+1)
&=\overline{\rho_{2n}}\Big[\overline{\rho_{2n-1}} u(2n-1)
    -\alpha_{2n-1}  u(2n)\Big]  \\
\nonumber
& \qquad \qquad   -\alpha_{2n}\Big[\overline{\alpha_{2n+1}}  u(2n+1)
    +\rho_{2n+1} u(2n+2)\Big].
\end{align}

Solutions to $\calE u=zu$ with $z \in \bbC^*$ and $u \in \bbC^\bbZ$ satisfy the iterative relation
\begin{equation}\label{eq:transfer_rel_A}
		\begin{bmatrix}u(2n+1)\\u(2n)\end{bmatrix}=M_{n,z}\begin{bmatrix}u(2n-1)\\u(2n-2)\end{bmatrix}, \quad n \in \bbZ,
\end{equation}
where the \emph{transfer matrices} $M_{n,z}$ are given by
\begin{align} \label{eq:GECMV_transmat}
	M_{n,z} & = \mathfrak{X}(\alpha_{2n}, \alpha_{2n-1}, \alpha_{2n-2},\rho_{2n}, \rho_{2n-1}, \rho_{2n-2}, z)\\
\nonumber
& :=\frac{1}{\rho_{2n}\rho_{2n-1}}\begin{bmatrix}
		z^{-1}+\alpha_{2n}\overline{\alpha}_{2n-1}+\alpha_{2n-1}\overline{\alpha}_{2n-2}+\alpha_{2n}\overline{\alpha}_{2n-2}z& -\overline{\rho}_{2n-2}\alpha_{2n-1}-\overline{\rho}_{2n-2}\alpha_{2n}z\\
		-\rho_{2n}\overline{\alpha}_{2n-1}-\rho_{2n}\overline{\alpha}_{2n-2}z&\rho_{2n}\overline{\rho}_{2n-2}z
	\end{bmatrix},
\end{align}\normalsize
for $n \in \bbZ$ and $z \in \bbC^*$. This follows from direct calculations using \eqref{eq:E_action_2n} and \eqref{eq:E_action_2n+1}, which are carried out in detail in \cite[Section~4]{CFO2023CMP}.
Note that the determinant of $M_{n,z}$ is given by
\begin{equation*}
    \det M_{n,z}=\frac{\overline{\rho_{2n-1}}\overline{\rho_{2n-2}}}{\rho_{2n}\rho_{2n-1}}.
\end{equation*}

Similar calculations give us the following iterative relations for starting points with reversed parity. For the reader's convenience, we sketch the argument below.
\begin{lemma}\label{lem:B}
    Suppose that $\calE u=zu$ for $z\in\bbC\backslash\{0\}$. Then
    \begin{equation} \label{eq:Bndefiningrelation}
		\begin{bmatrix}u(2n+2)\\u(2n+1)\end{bmatrix}=N_{n,z}\begin{bmatrix}u(2n)\\u(2n-1)\end{bmatrix}, \quad n \in \bbZ,
    \end{equation}
    where
    \begin{align}
\label{eq:Ndef}
        N_{n,z}
&  = \mathfrak{X}(\overline{\alpha_{2n+1}}, \overline{\alpha_{2n}}, \overline{\alpha_{2n-1}},\rho_{2n+1}, \rho_{2n}, \rho_{2n-1}, z^{-1}) \\
\nonumber
& =\frac1{\rho_{2n+1}\rho_{2n}}
        \begin{bmatrix}
            z+\overline{\alpha_{2n}}\alpha_{2n-1}+\overline{\alpha_{2n+1}}\alpha_{2n-1}z^{-1}+\overline{\alpha_{2n+1}}\alpha_{2n} & -\overline{\alpha_{2n}}\overline{\rho_{2n-1}}-\overline{\alpha_{2n+1}}\overline{\rho_{2n-1}}z^{-1} \\
            -\alpha_{2n}\rho_{2n+1}-\alpha_{2n-1}\rho_{2n+1}z^{-1} & \overline{\rho_{2n-1}}\rho_{2n+1}z^{-1}
        \end{bmatrix}.
    \end{align}
    Moreover,
    \begin{equation*}
        \det N_{n,z}=\frac{\overline{\rho_{2n}}\overline{\rho_{2n-1}}}{\rho_{2n+1}\rho_{2n}}.
    \end{equation*}
\end{lemma}
\begin{proof}
    From $\alpha_{2n}\eqref{eq:E_action_2n}+\rho_{2n}\eqref{eq:E_action_2n+1}$ and $\overline{\rho_{2n}}\eqref{eq:E_action_2n}-\overline{\alpha_{2n}}\eqref{eq:E_action_2n+1}$ we get
    \begin{equation}\label{eq:B_1}
        \alpha_{2n}zu(2n)+\rho_{2n}zu(2n+1)=\overline{\rho_{2n-1}}u(2n-1)  -\alpha_{2n-1} u(2n)
    \end{equation}
    and
    \begin{equation}\label{eq:B_2}
        \overline{\rho_{2n}}zu(2n)-\overline{\alpha_{2n}}zu(2n+1)=\overline{\alpha_{2n+1}} u(2n+1)+\rho_{2n+1} u(2n+2),
    \end{equation}
    respectively, where we used that $|\alpha|^2+|\rho|^2=1$. Solving the first equation for $u(2n+1)$ and inserting it into the second one yields the relations in \eqref{eq:Bndefiningrelation}. The determinant is determined by a straightforward calculation.
\end{proof}

Lemma~\ref{lem:B} immediately implies that for two solutions $u,v$ of the generalized eigenvalue equation \eqref{eq:generalizedEvalEq}, we have
\begin{equation}
    \begin{bmatrix} u(2n+2) & v(2n+2) \\ u(2n+1) & v(2n+1) \end{bmatrix}=N_{n,z}N_{n-1,z}\dots N_{0,z}\begin{bmatrix} u(0) & v(0) \\ u(-1) & v(-1) \end{bmatrix}.
\end{equation}
Taking determinants on both sides and rearranging terms gives
\begin{align*}
    W(u,v)(n)&:=\rho_{2n+1} \frac{\rho_{2n}\rho_{2n-1}\dots\rho_1\rho_{0}}{\overline{\rho_{2n}}\overline{\rho_{2n-1}}\dots\overline{\rho_{1}}\overline{\rho_{0}}}(u(2n+2)v(2n+1)- u(2n+1)v(2n+2)) \\
    &=\overline{\rho_{-1}}(u(0)v(-1)- u(-1)v(0)).
\end{align*}
We shall call $W(u,v)$ the \emph{Wronskian} of $u$ and $v$. 

\begin{remark} \label{rem:symmForRhoFromAlpha}
The following observation will be helpful in the proof: for $\rho_n^\bbC:= i \sqrt{1-|\alpha_n|^2}$, then 
\begin{align*}
\left| \rho_{n}^\bbC - \left(- \overline{\rho_m^\bbC}\right) \right|
& = \left|(1-|\alpha_{n}|^2)^{1/2} -  (1-|\alpha_m|^2)^{1/2} \right| \\
& \leq \frac{2r}{\sqrt{1-r^2}} \big||\alpha_{n}| - |\alpha_m| \big|,
\end{align*}
where $r = \max\{|\alpha_n|, |\alpha_m|\}$. Thus, if $\alpha$ is $(B,\zeta)$-reflective and satisfies \eqref{eq:VCbd}, then $\rho^\bbC$ obeys
\begin{equation} \label{eq:rhosymm}
\left| (\calR_\zeta \rho^\bbC)_n - \left(- \overline{\rho_n^\bbC}\right) \right|
\leq \frac{2r_0}{\sqrt{1-r_0^2}} e^{-B|\zeta|}
\end{equation}
for all $n \in \bbZ$ such that $|\zeta - n| \leq \exp(B|\zeta|)$.
\end{remark}

In view of the foregoing remark, we introduce the following:

\begin{notation}
Throughout the argument, we let $C>0$ stand for a constant that depends only on $r_0$ from \eqref{eq:VCbd}, and write $f \lesssim g$ to mean $f \leq Cg$ for such a constant.
\end{notation}

\begin{proof}[Proof of Theorem~\ref{t:JitomirskayaSimon:refl}]
Assume $\alpha$ is $B$-reflective and $B$ is sufficiently large. Choose $|\zeta_i| \to \infty$ such  that $\alpha$ is $(B,\zeta_i)$-reflective. By passing to a subsequence, we may assume that every $\zeta_i$ has the same sign and $2\zeta_i$ has the same residue modulo $4$ for every $i$. Without loss of generality, we consider the case $\zeta_i>0$ and $2\zeta_i \cong 2 \ \mathrm{mod} \ 4$ for every $i$, so we write $\zeta_i = 2m_i- 1$ with $m_i \in \bbN$. The other cases are similar. The main difference between the cases is what transfer matrix cocycle one must use. To be concrete, the current case compares $N$ with its ``reflected'' variant, while the case $2\zeta_i \cong 0 \ \mathrm{mod} \ 4$ compares $M$ with its reflected variant. The case of half-integer centers ($2\zeta_i$ odd) requires one to use transfer matrices associated with the \emph{transpose} of a CMV matrix, but aside from that, all steps are analogous.

Due to \cite[Theorem 2.1]{CFLOZ}, it suffices to prove the statement of the theorem for a single admissible $\rho$. As discussed in Remark~\ref{rem:symmForRhoFromAlpha}, the symmetry of the $\alpha_n$ can be passed to $\rho_n$, and furthermore, we need $\rho$ to satisfy bounds such as \eqref{eq:rhosymm} (up to shifting the center of reflection). Thus, we choose $\rho_n= i\sqrt{1-|\alpha_n|^2}$. Consequently, we may assume henceforth that
	\begin{equation}\label{eq:assump}
		|\rho_{4m_i- \ell-2}+\overline{\rho_\ell}| \lesssim  e^{-B\zeta_i},\qquad |\alpha_{4m_i-\ell-2}-\overline{\alpha_\ell}|<e^{-B\zeta_i},
	\end{equation}
for every $\ell \in \bbZ$ satisfying $|\zeta_i-\ell| \leq e^{B\zeta_i}$, and recall that our assumptions give $|\rho_\ell | \geq c_0>0$ uniformly in $\ell \in \bbZ$.

\begin{remark}
	By a standard telescoping estimate, \eqref{eq:assump} implies that (since $|\rho|,|\alpha|\leq1$)
	\begin{align*}
		|\alpha_{4m_i-\ell-2}\rho_{4m_i-\ell'-2}+\overline{\alpha_\ell}\overline{\rho_{\ell'}}|
		& \lesssim   e^{-B\zeta_i}
	\end{align*}
for all $\ell,\ell' \in \bbZ$ with $|\zeta_i-\ell|, |\zeta_i-\ell'| \leq e^{B\zeta_i}$,
with similar bounds for other suitable combinations of $\rho$'s and $\alpha$'s.
\end{remark}

Let us define the reflection of $u$ through $2m_i-\tfrac12$ by
\begin{equation}\label{eq:ui}
	u_i(k) = u(4m_i-k-1)
\end{equation}
and note that
\begin{equation*}
	u_i(2n-1)=u(4m_i-2n),\qquad u_i(2n)=u(4m_i-2n-1).
\end{equation*}

\begin{lemma}\label{lem:W-W}
	Under the assumptions \eqref{eq:assump} 
and with $u_i$ as in \eqref{eq:ui},
	\begin{equation*}
		|W(u,u_i)(n)-W(u,u_i)(n')|  \lesssim  e^{-B\zeta_i}
	\end{equation*}
for all $n,n'$ with $|m_i-n|, |m_i-n'| \leq \tfrac12 e^{B\zeta_i}-1$.
\end{lemma}

\begin{proof}
To begin, notice that the assumption on $|m_i-n|$ implies that for $\eta = -1,0,1$, one has
\begin{equation*}
	|\zeta_i-(2n+\eta)| = |2m_i-1-2n-\eta| \leq 2|m_i-n|+2\leq e^{B\zeta_i}, 
\end{equation*}
and thus we can (and will) apply \eqref{eq:assump} with $\ell = 2n-1, 2n, 2n+1$ in the present argument.

	By definition of the Wronskian we have
	\begin{align*}
		|W(u,v)(n)-W(u,v)(n-1)|
		&=\frac1{|\overline{\rho_{2n}}|}\Big|\rho_{2n+1}\rho_{2n}\Big[u(2n+2)v(2n+1)- u(2n+1)v(2n+2)\Big]\\
		&\qquad-\overline{\rho_{2n}}\overline{\rho_{2n-1}}\Big[u(2n)v(2n-1)- u(2n-1)v(2n)\Big]\Big|.
	\end{align*}
	Applying the generalized eigenequation \eqref{eq:E_action_2n} and \eqref{eq:E_action_2n+1} for $u$ yields
\small
	\begin{align}\nonumber
		&|W(u,v)(n)-W(u,v)(n-1)|\\ 
		& \nonumber \qquad=\frac1{|\overline{\rho_{2n}}|}\Big|\Big[zu(2n)-
		\overline{\alpha_{2n}}\big[\overline{\rho_{2n-1}}u(2n-1)-\alpha_{2n-1}u(2n)\big]-\rho_{2n}\overline{\alpha_{2n+1}}u(2n+1)\Big]v(2n+1)\\
		& \nonumber \qquad\quad-\rho_{2n+1}\rho_{2n}u(2n+1)v(2n+2)-\overline{\rho_{2n}}\overline{\rho_{2n-1}}u(2n)v(2n-1)\\
		& \nonumber \qquad\quad+\Big[zu(2n+1)+\overline{\rho_{2n}}\alpha_{2n-1} u(2n)+\alpha_{2n}\big[\overline{\alpha_{2n+1}}  u(2n+1)
		+\rho_{2n+1} u(2n+2)\big]\Big]v(2n)\Big|
		\\
		&\nonumber \qquad=\frac1{|\overline{\rho_{2n}}|}\Big|\Big[\overline{\rho_{2n}}\alpha_{2n-1}v(2n)-\overline{\rho_{2n}}\overline{\rho_{2n-1}}v(2n-1)\Big]u(2n)-\Big[\rho_{2n}\overline{\alpha_{2n+1}}v(2n+1)+\rho_{2n+1}\rho_{2n}v(2n+2)\Big]u(2n+1)\\
		&\nonumber\qquad\quad
		-\overline{\alpha_{2n}}\Big[\overline{\rho_{2n-1}}u(2n-1)-\alpha_{2n-1}u(2n)\Big]v(2n+1)+\alpha_{2n}\Big[\overline{\alpha_{2n+1}}  u(2n+1)
		+\rho_{2n+1} u(2n+2)\Big]v(2n)\\
		&\qquad\quad
		+z\Big[u(2n)v(2n+1)+u(2n+1)v(2n)\Big]\Big|.
\label{eq:WronskiCalcSetup}
	\end{align}\normalsize
	Now setting $v=u_i$ with $u_i$ as in \eqref{eq:ui}, we claim:
	\begin{equation} \begin{split} \label{eq:longWronskiCalc}
		&|W(u,u_i)(n)-W(u,u_i)(n-1)|\\
	&\qquad=\frac1{|\overline{\rho_{2n}}|}\Big|\Big[\overline{\rho_{2n}}\alpha_{2n-1}+\rho_{4m_i-2n-2}\overline{\alpha_{4m_i-2n-1}}\Big]u(4m_i-2n-1)u(2n)\\
	&\qquad\quad-\Big[\overline{\rho_{2n}}\overline{\rho_{2n-1}}-\rho_{4m_i-2n-2}\rho_{4m_i-2n-1}\Big]u(4m_i-2n)u(2n)\\
	&\qquad\quad-\Big[\rho_{2n}\overline{\alpha_{2n+1}}+\overline{\rho_{4m_i-2n-2}}\alpha_{4m_i-2n-3}\Big]u(4m_i-2n-2)u(2n+1)\\
	&\qquad\quad+\Big[\rho_{2n+1}\rho_{2n}-\overline{\rho_{4m_i-2n-2}}\overline{\rho_{4m_i-2n-3}}\Big]u(4m_i-2n-3)u(2n+1)\\	
	&\qquad\quad+z\Big[\overline{\alpha_{4m_i-2n-2}}\alpha_{4m_i-2n-2}-\overline{\alpha_{2n}}\alpha_{2n}\Big]u(4m_i-2n-2)u(2n)\\
	&\qquad\quad+z\Big[\overline{\alpha_{4m_i-2n-2}}\rho_{4m_i-2n-2}+\overline{\rho_{2n}}\alpha_{2n}\Big]u(4m_i-2n-1)u(2n)\\
	&\qquad\quad-z\Big[\alpha_{4m_i-2n-2}\overline{\rho_{4m_i-2n-2}}+\overline{\alpha_{2n}}\rho_{2n}\Big]u(4m_i-2n-2)u(2n+1)\\
	&\qquad\quad+z\Big[\alpha_{4m_i-2n-2}\overline{\alpha_{4m_i-2n-2}}-\overline{\alpha_{2n}}\alpha_{2n} \Big]u(4m_i-2n-1))u(2n+1)\Big|.
	\end{split}\end{equation}
The proof of \eqref{eq:longWronskiCalc} is not straightforward and also somewhat lengthy, so we put it into the Appendix.
Now, using the assumptions \eqref{eq:assump}, \eqref{eq:longWronskiCalc}, the Cauchy--Schwarz inequality, and $z \in \partial \bbD$, we get
\begin{equation}
\big|W(u,u_i)(n) - W(u,u_i)(n')\big|
\lesssim e^{-B\zeta_i} 
\end{equation}
for all $n,n'$ as in the statement of the Lemma.\end{proof}

Next, use the Cauchy--Schwarz inequality and  $|\rho_k|\leq1$ to see that
\begin{align} \label{eq:Well1bound}
	\sum_n\big|W(u,u_i)(n)\big|\leq2.
\end{align}
This implies that there exists $n\in\bbZ$ with $|m_i-n|\leq \frac{1}{2}e^{B\zeta_i}-1$ such that $|W(u,u_i)(n)|\leq 3e^{-B\zeta_i}$ (for if not, summing $W(u,u_i)(n)$ over the relevant range of $n$'s produces a contradiction to \eqref{eq:Well1bound}).
Lemma \ref{lem:W-W} then implies that 
\begin{equation}
|W(u,u_i)(n)|  
\lesssim  e^{-B\zeta_i} \text{ for all } n\in\bbZ \text{ with }|m_i-n|\leq \frac{1}{2} e^{B\zeta_i}-1.
\end{equation}

Let us define $u_i^\pm=u\pm u_i$, and (using \eqref{eq:ui}) note that 
\begin{equation}\label{eq:ui_props}
	u_i^\pm(2m_i-1)=
	\pm u_i^\pm(2m_i).
\end{equation}
Further, we define
\begin{equation} \label{eq:Phiipmdef}
	\Phi^\pm_i(n)=\begin{bmatrix}u_i^\pm(2n)\\u_i^\pm(2n-1)\end{bmatrix}.
\end{equation}
\begin{lemma} \label{lem:si_choice}
	For each $i$, there exists $s=s_i \in \{+,-\}$ such that
	\begin{equation}\label{eq:Phi_mi_bdd}
		\|\Phi^s_i(m_i)\|=\Big[|u_i^s(2m_i)|^2+|u_i^s(2m_i-1)|^2\Big]^{1/2} \lesssim e^{-B\zeta_i/2}.
	\end{equation}
\end{lemma}
\begin{proof}
	First note that $W(u_i^-,u_i^+)=2W(u,u_i)$. Together with \eqref{eq:ui_props} this implies that
	\begin{align*}
2|\rho_{2m_i-1}|\big|u_i^-(2m_i-1)u_i^+(2m_i)\big|
& = |\rho_{2m_i-1}|\big|u_i^-(2m_i)u_i^+(2m_i-1)-u_i^-(2m_i-1)u_i^+(2m_i)\big|\\
& =		|W(u_i^-,u_i^+)(m_i-1)| \\
& = 2|W(u,u_i)(m_i-1)| \\	
		& \lesssim e^{-B\zeta_i}.
	\end{align*}
	Thus
	\begin{equation*}
		\big|u_i^-(2m_i-1)u_i^+(2m_i)\big|  \lesssim  e^{-B\zeta_i},
	\end{equation*}
so either
	\begin{equation*}
	\big|u_i^-(2m_i-1)\big|  \lesssim  e^{-B\zeta_i/2} \qquad
\text{ or }\qquad
\big|u_i^+(2m_i)\big|  \lesssim e^{-B\zeta_i/2}.
	\end{equation*}
	In either case, \eqref{eq:Phi_mi_bdd} follows by symmetry of the $u_i^\pm$, that is, by \eqref{eq:ui_props}.	
\end{proof}

Next, we want to show that a similar bound holds for $\Phi_i^\pm(0)$. To this end, define
\begin{equation*}
	\Phi(n)=\begin{bmatrix}u(2n)\\u(2n-1)\end{bmatrix},\qquad \Phi_i(n)=\begin{bmatrix}u_i(2n)\\u_i(2n-1)\end{bmatrix}=\begin{bmatrix}u(4m_i-2n-1)\\u(4m_i-2n)\end{bmatrix}
\end{equation*}
and note that $\Phi_i^\pm=\Phi\pm\Phi_i$ on account of \eqref{eq:Phiipmdef} and the definition of $u_i^\pm$. Let us write
\begin{equation} \label{eq:N^ndef}
	N^n_z=N^{[n-1,0]}_z=\prod_{k=n-1}^{0}N_{k,z},
\end{equation}
where $N$ is given by \eqref{eq:Ndef}.
By Lemma \ref{lem:B} we have $\Phi(n)=N^n_{z}\Phi(0)$ and thus $\Phi(0)=[N^n_{z}]^{-1}\Phi(n)$ with

\begin{equation*}
	[N^{n}_{z}]^{-1}=\prod_{k=0}^{n-1}N_{k,z}^{-1}.
\end{equation*}
The reader can check that $N_{n,z}^{-1}$ is given by
\begin{align}
\nonumber
	N_{n,z}^{-1}&=\frac1{\overline{\rho_{2n}}\overline{\rho_{2n-1}}}
	\begin{bmatrix}
		\overline{\rho_{2n-1}}\rho_{2n+1}z^{-1} & \overline{\alpha_{2n}}\overline{\rho_{2n-1}}+\overline{\alpha_{2n+1}}\overline{\rho_{2n-1}}z^{-1} \\
		\alpha_{2n}\rho_{2n+1}+\alpha_{2n-1}\rho_{2n+1}z^{-1} & z+\overline{\alpha_{2n}}\alpha_{2n-1}+\overline{\alpha_{2n+1}}\alpha_{2n-1}z^{-1}+\overline{\alpha_{2n+1}}\alpha_{2n}
	\end{bmatrix}\\
	\label{eq:B_nz_inverse}
	&=:\frac1{\overline{\rho_{2n}}\overline{\rho_{2n-1}}}\widetilde N_{n,z}^{-1}.
\end{align}

Also, with the ``mirrored transfer'' matrices defined by
\begin{align}
\nonumber
	\Ni_{n,z}&:=\begin{bmatrix}0&1\\1&0\end{bmatrix}N_{n,z}\begin{bmatrix}0&1\\1&0\end{bmatrix}\\
	\nonumber
	&=\frac1{\rho_{2n+1}\rho_{2n}}
	\begin{bmatrix}
		 \overline{\rho_{2n-1}}\rho_{2n+1}z^{-1} & -\alpha_{2n}\rho_{2n+1}-\alpha_{2n-1}\rho_{2n+1}z^{-1} \\
		-\overline{\alpha_{2n}}\overline{\rho_{2n-1}}-\overline{\alpha_{2n+1}}\overline{\rho_{2n-1}}z^{-1} & z+\overline{\alpha_{2n}}\alpha_{2n-1}+\overline{\alpha_{2n+1}}\alpha_{2n-1}z^{-1}+\overline{\alpha_{2n+1}}\alpha_{2n}
	\end{bmatrix}\\
	\label{eq:rotB_nz}
	&=:\frac1{\rho_{2n+1}\rho_{2n}} \widetilde \Ni_{n,z},
\end{align}
we deduce from Lemma~\ref{lem:B} that 
\begin{equation*}
	\Phi_i(0)=\begin{bmatrix}u_i(0)\\u_i(-1)\end{bmatrix}=\begin{bmatrix}u(4m_i-1)\\u(4m_i)\end{bmatrix}= \underbrace{\prod_{k=2m_i-1}^{m_i}\Ni_{k,z}}_{=:\Ni_{z}^{[2m_i-1,m_i]}}\begin{bmatrix}u(2m_i-1)\\u(2m_i)\end{bmatrix}=\Ni_{z}^{[2m_i-1,m_i]}\Phi_i({m_i}).
\end{equation*}

\begin{lemma}\label{lem:Phi0->0}
With $s_i$ as in Lemma~\ref{lem:si_choice}, $\big\|\Phi^{s_i}_i(0)\big\|\to0$ as $m_i\to\infty$.
\end{lemma}
\begin{proof}
	We have
	\begin{align}
\nonumber
		\Phi^\pm_i(0)	&=\Phi(0)\pm\Phi_i(0)=[N^{m_i}_{z}]^{-1}\Phi({m_i})\pm\Ni^{[2m_i-1,m_i]}_z\Phi_i({m_i})\\
\label{lem:25:Phiipm0setup}
						&=[N^{m_i}_{z}]^{-1}\Phi_i^\pm({m_i})\mp\Big[[N^{m_i}_{z}]^{-1}-\Ni^{[2m_i-1,m_i]}_{z}\Big]\Phi_i({m_i}).
	\end{align}	
	Thus, to control $\big\|\Phi^\pm_i(0)\big\|$ we need to bound $\big\|[N^{m_i}_{z}]^{-1}\big\|$ and $\big\|[N^{m_i}_{z}]^{-1}-\Ni^{[2m_i-1,m_i]}_{z}\big\|$. 
To do this, we use the standard telescoping estimate: for matrices $A_1,\dots,A_n$ and $B_1,\dots,B_n$ we have that
	\begin{equation*}
		\prod_{k=n}^1A_k-\prod_{k=n}^1B_k=\sum_{i=n}^1\left[\prod_{k=n}^{i+1} A_k\right]\left[A_i-B_i\right]\left[\prod_{k=i-1}^1 B_k\right]
	\end{equation*}
	and thus
	\begin{equation*}
		\Big\|\prod_{k=n}^1A_k-\prod_{k=n}^1B_k\Big\|\leq\sum_{i=n}^1\Big\|\prod_{k=n}^{i+1} A_k\Big\|\:\big\|A_i-B_i\big\|\:\Big\|\prod_{k=i-1}^1 B_k\Big\|.
	\end{equation*}
	Applying this to the second term of \eqref{lem:25:Phiipm0setup}, two of the factors are bounded by the submultiplicativity of the operator norm and the observation that 
\begin{equation}
\big\| N_{k,z}^{-1}\big\|, \ \big\| \Ni_{2m_i-k-1,z}\big\| \leq C.
\end{equation}

Comparing the expressions for $ N^{-1}$ and $ \Ni$ in \eqref{eq:B_nz_inverse} and \eqref{eq:rotB_nz} and using the symmetries \eqref{eq:assump} gives
\begin{equation}
\big\|N_{k,z}^{-1} - \Ni_{2m_i-k-1,z}\big\|\lesssim e^{-B\zeta_i}.
\end{equation}

Putting it all together, we have 
\begin{equation} \label{lem:25:est1}
\big\|[N^{m_i}_{z}]^{-1}-\Ni^{[2m_i-1,m_i]}_{z}\big\|
\leq  m_i C^{m_i} e^{-B\zeta_i}
\end{equation}
Similarly, 
\begin{equation}\label{lem:25:est2}
\big\|[N^{m_i}_{z}]^{-1}\Phi_i^{s_i}(m_i)\big\|\leq C^{m_i} e^{-B\zeta_i/2}.
\end{equation}
Since the right-hand sides of \eqref{lem:25:est1} and \eqref{lem:25:est2} go to zero as $i \to \infty$ when $B$ is large enough, combining these estimates with \eqref{lem:25:Phiipm0setup} proves the lemma.
\end{proof}
With this result we can conclude the proof of Theorem~\ref{t:JitomirskayaSimon:refl} as follows.	Lemma \ref{lem:Phi0->0} implies that $\|\Phi(0)\|-\|\Phi(m_i)\|\to0$. By assumption, $u\in\ell^2(\bbZ)$, so we conclude that $\Phi(0)=0$, which implies  that $u=0$ (by \eqref{eq:B_1}). This contradicts normalization of $u$.
\end{proof}

Let us point out that Corollary~\ref{coro:JitomirskayaSimon:main} now follows in short order by a soft argument with the Baire category theorem:

\begin{proof}[Proof of Corollary~\ref{coro:JitomirskayaSimon:main}]
Assume $\alpha$ satisfies \eqref{eq:VCbd} and is reflection-symmetric, write $\Omega = \hull(\alpha)$, and fix $B$ large. 

For each $m \in \bbZ$, let $\refl(B,m) \subseteq \Omega$ denote the set of $(B,m)$-reflective elements of $\Omega$.
For each $k$, the set
\[\bigcup_{m \geq k} \refl(B,m)\]
is open by definition and dense in $\Omega$ (as it contains a translation semiorbit of $\alpha$). Therefore 
\[
\Omega_0
= \bigcap_{k\geq 1}\bigcup_{m \geq k} \refl(B,m)
\]
 is a dense $G_\delta$ by the Baire category theorem. Since $\calE_\omega$ has purely continuous spectrum for every $\omega \in \Omega_0$ by Theorem~\ref{t:JitomirskayaSimon:refl}, we are done.
\end{proof}

We now move on to a discussion of the quantiative strengthening in Theorem~\ref{t:JL:quant}. If $\alpha$ is almost-periodic, its hull $\Omega = \hull(\alpha)$ is a compact abelian topological group and hence is equipped with a unique translation-invariant Borel probability measure, which we denote by $\mu$. For each $\omega \in \Omega$, we may consider the associated transfer matrices $N_{k,z} = N_{k,z}(\omega)$ and their iterates $N^{n}_z(\omega)$ as in \eqref{eq:Ndef} and \eqref{eq:N^ndef}. The corresponding Lyapunov exponent is given by\footnote{The extra factor of one-half is so that the Lyapunov exponent of this cocycle matches that of the usual Szeg\H{o} and Gestesy--Zinchenko cocycles, on account of the connections between cocycles discussed in the literature, e.g., \cite{CFLOZ, DFO2016JMPA, yang2022localization}.}
\begin{equation} \label{eq:LEdef}
L(z) = \lim_{n\to\infty} \frac{1}{2n} \int_\Omega \log\|N^n_z(\omega)\| \, d\mu(\omega).
\end{equation}
The argument is an adaptation of the previous argument along the contours of \cite{JitoLiu2023JEMS}; since it is so similar to the previous argument, we only sketch the main steps.
\begin{proof}[Proof Sketch of Theorem~\ref{t:JL:quant}]
Fix $\alpha$ almost-periodic and $B$-reflective satisfying \eqref{eq:VCbd}, assume $B > 2 L(z)$, and fix $\varepsilon>0$ small enough that $2(L(z)+\varepsilon) <B$. The initial parts of the proof proceed in precisely the same manner until one reaches \eqref{lem:25:Phiipm0setup}. At this point, one estimates the terms using the semiuniform ergodic theorem of Furman (instead of a crude \textit{a priori} esimate) to get improvements to \eqref{lem:25:est1} and \eqref{lem:25:est2}:
\begin{equation}
\begin{split}
\|[N^{m_i}_{z}]^{-1} \Phi_i^{s_i}(m_i)\| 
\lesssim e^{2m_i(L(z)+\varepsilon)}e^{-B\zeta_i/2}, \\   
\Big\| [N^{m_i}_{z}]^{-1}-\Ni^{[2m_i-1,m_i]}_{z}\Big\| \lesssim m_i e^{2m_i(L(z)+\varepsilon)} e^{-B\zeta_i}.
\end{split}
\end{equation}
Recalling that $\zeta_i = 2m_i-1$, the condition $B> 2(L(z) + \varepsilon)$ suffices to see that the expression in \eqref{lem:25:Phiipm0setup} goes to zero and hence the rest of the argument can be run as before.
\end{proof}
\begin{proof}[Proof of Corollary~\ref{coro:genQP}]
Notice that the assumptions imply that $\alpha_\omega$ satisfies \eqref{eq:VCbd} for all $\omega$.
Let $\delta = \delta(\beta,\omega)$. For any $\varepsilon>0$, we may choose $|k_i|\to \infty$ with $\vertiii{2\omega+k_i\beta}< e^{(\delta-\varepsilon)|k_i|}$, and observe by symmetry of $f$ that
\begin{align*}
 \Big|\alpha_\omega(k_i-n) - \overline{\alpha_\omega(n)}\Big|
& = \Big|f(\omega+(k_i-n)\beta) - \overline{f(\omega+n\beta)}\Big| \\
& =  |f(\omega+(k_i-n)\beta) - f(-\omega-n\beta)| \\
& \lesssim \vertiii{2\omega+k_i\beta} \\
&\lesssim e^{-(\delta-\varepsilon)|k_i|}.
\end{align*}
Thus, if $L(z) <\delta$, then we can choose $\varepsilon>0$ with $L(z) < \delta-\varepsilon$ and it follows from the above that $\alpha_\omega$ is $B=2(\delta-\varepsilon)$-reflective,\footnote{Notice that $k_i/2$, not $k_i$, is the center of reflection, which accounts for the factor of $2$ here.} so $z$ is not an eigenvalue of $\calE_\omega$ by Theorem~\ref{t:JL:quant}.
\end{proof}

With the main results proved, we may now reap our harvest of corollaries. 
\begin{proof}[Proof of Corollary~\ref{coro:uamo}]
Notice that $\calE_{\lambda_1,\lambda_2, \Phi,\theta}^{\rm UAMO}$ falls into the setting of Corollary~\ref{coro:genQP} with 
\begin{equation} \label{eq:UAMOdynamicalDef}
\Omega = \bbT \times \bbZ_2,
\quad \beta = (\Phi/2,1), \quad
f(\theta,j)
= \begin{cases}
\lambda_2 \cos(2\pi \theta) & j = 0\\
(1-\lambda_1^2)^{1/2}     & j = 1.
\end{cases}
\end{equation} 
Clearly, $f$  is Lipschitz continuous and even. Thus, according to Corollary~\ref{coro:genQP} together with \cite[Corollary~2.10]{CFO2023CMP}, absence of eigenvalues holds throughout the spectrum for any $\theta$ such that
\begin{equation} \label{eq:UAMOdeltablockLB}
\delta((\Phi/2,1),(\theta,0))>
\log\left[ \frac{\lambda_2(1+\lambda_1')}{\lambda_1(1+\lambda_2')}  \right].
\end{equation}
Recalling the definition of $\delta$ from \eqref{eq:deltaBetaOmegaDef}, we see that $\delta((\Phi/2,1),(\theta,0)) = \frac{1}{2}\delta(\Phi,\theta)$, so \eqref{eq:UAMOdeltablockLB} is equivalent to $\frac{1}{2} \delta(\Phi,\theta)> L(z)$. Since this holds for generic $\theta \in \bbT$, this shows absence of eigenvalues when \eqref{eq:deltaBoundUAMO} holds. Since the expression on the right-hand side of \eqref{eq:deltaBoundUAMO} is equal to the Lyapunov exponent on the spectrum \cite{CFO2023CMP} and is positive when $\lambda_1<\lambda_2$, it follows that the absolutely continuous and point parts of the spectrum are both absent when \eqref{eq:deltaBoundUAMO} and $\lambda_1<\lambda_2$ hold.

\end{proof}

\begin{proof}[Proof of Corollary~\ref{coro:mosaic}]
This follows from similar considerations as in the previous corollary after changing $\bbZ_2$ to $\bbZ_{2s}$ and suitably modifying $f$.
\end{proof}

\begin{proof}[Proof of Corollary~\ref{coro:palin}]
Again, observe that the assumptions imply that $\alpha_\omega$ satisfies \eqref{eq:VCbd} for all $\omega$. This follows immediately by combining Theorem~\ref{t:JitomirskayaSimon:refl} with work of Hof--Knill--Simon, specifically \cite[Proposition~2.1]{HofKniSim1995CMP}. 
\end{proof}

\begin{appendix}
\section{Proof of \eqref{eq:longWronskiCalc}}
We need one preliminary: the $\mathcal{LM}$ factorization. Concretely, we define 
\begin{equation}\label{eq:theta_mat}
\mathcal{L}=\bigoplus_{n\in\bbZ}\Theta(\alpha_{2n},\rho_{2n}), \quad \mathcal{M}=\bigoplus_{n\in\bbZ}\Theta(\alpha_{2n+1},\rho_{2n+1}), \quad	\Theta(\alpha,\rho)=\begin{bmatrix}\overline{\alpha}&\rho\\\overline{\rho}&-\alpha\end{bmatrix},
\end{equation}
where $\Theta(\alpha_j,\rho_j)$ acts on the subspace $\ell^2(\{j,j+1\})$. With these definitions, one has $\calE = \mathcal{L}\mathcal{M}$. Since $\calL$ and $\calM$ are themselves unitary, this means that $\calE u = zu$ is equivalent to $\calM u = z\calL^*u$. Evaluating this relation at coordinates $2n$ and $2n+1$ gives
\begin{align}\label{eq:LMeig1}
\overline{\rho_{2n-1}} u(2n-1) - \alpha_{2n-1}u(2n) & = z(\alpha_{2n}u(2n)+ \rho_{2n} u(2n+1)) \\
\label{eq:LMeig2}
\overline{\alpha_{2n+1}}u(2n+1) + \rho_{2n+1}u(2n+2) & = z(\overline{\rho_{2n}}u(2n)-\overline{\alpha_{2n}} u(2n+1))
\end{align}

We now proceed with the proof of \eqref{eq:longWronskiCalc}. Starting from \eqref{eq:WronskiCalcSetup}, using the eigenequation of $u$, and adding $0$ in an inspired manner gives:
\footnotesize
	\begin{align*}
		&|W(u,u_i)(n)-W(u,u_i)(n-1)|\\
	&\qquad=\frac1{|\overline{\rho_{2n}}|}\Big|\Big[\overline{\rho_{2n}}\alpha_{2n-1}u(4m_i-2n-1)-\overline{\rho_{2n}}\overline{\rho_{2n-1}}u(4m_i-2n)\Big]u(2n)\\
	&\qquad\quad+\rho_{4m_i-2n-2}\Big[\overline{\alpha_{4m_i-2n-1}} u(4m_i-2n-1)+\rho_{4m_i-2n-1} u(4m_i-2n)\Big]u(2n)\\
	&\qquad\quad-\rho_{4m_i-2n-2}\Big[\overline{\alpha_{4m_i-2n-1}} u(4m_i-2n-1)+\rho_{4m_i-2n-1} u(4m_i-2n)\Big]u(2n)\\
	&\qquad\quad-\Big[\rho_{2n}\overline{\alpha_{2n+1}}u(4m_i-2n-2)+\rho_{2n+1}\rho_{2n}u(4m_i-2n-3)\Big]u(2n+1)\\
	&\qquad\quad+\Big[\overline{\rho_{4m_i-2n-2}}(\overline{\rho_{4m_i-2n-3}} u(4m_i-2n-3)
	-\alpha_{4m_i-2n-3}  u(4m_i-2n-2))\Big]u(2n+1)\\		&\qquad\quad-\Big[\overline{\rho_{4m_i-2n-2}}(\overline{\rho_{4m_i-2n-3}} u(4m_i-2n-3)
	-\alpha_{4m_i-2n-3}  u(4m_i-2n-2))\Big]u(2n+1)\\
	&\qquad\quad
	-\overline{\alpha_{2n}}\Big[\overline{\rho_{2n-1}}u(2n-1)-\alpha_{2n-1}u(2n)\Big]u(4m_i-2n-2)\\
	&\qquad\quad+\Big[\alpha_{2n}(\overline{\alpha_{2n+1}}  u(2n+1)
	+\rho_{2n+1} u(2n+2))\Big]u(4m_i-2n-1)\\
	&\qquad\quad
	+z(u(2n)u(4m_i-2n-2)+u(2n+1)u(4m_i-2n-1))\Big|\\
	&\qquad=\frac1{|\overline{\rho_{2n}}|}\Big|\Big[(\overline{\rho_{2n}}\alpha_{2n-1}+\rho_{4m_i-2n-2}\overline{\alpha_{4m_i-2n-1}})u(4m_i-2n-1)-(\overline{\rho_{2n}}\overline{\rho_{2n-1}}-\rho_{4m_i-2n-2}\rho_{4m_i-2n-1})u(4m_i-2n)\Big]u(2n)\\
	&\qquad\quad-\rho_{4m_i-2n-2}\Big[\overline{\alpha_{4m_i-2n-1}} u(4m_i-2n-1)+\rho_{4m_i-2n-1} u(4m_i-2n)\Big]u(2n)\\
	&\qquad\quad-\Big[(\rho_{2n}\overline{\alpha_{2n+1}}+\overline{\rho_{4m_i-2n-2}}\alpha_{4m_i-2n-3})u(4m_i-2n-2)+(\rho_{2n+1}\rho_{2n}-\overline{\rho_{4m_i-2n-2}}\overline{\rho_{4m_i-2n-3}})u(4m_i-2n-3)\Big]u(2n+1)\\
	&\qquad\quad-\Big[\overline{\rho_{4m_i-2n-2}}(\overline{\rho_{4m_i-2n-3}} u(4m_i-2n-3)
	-\alpha_{4m_i-2n-3}  u(4m_i-2n-2))\Big]u(2n+1)\\
	&\qquad\quad
	-\overline{\alpha_{2n}}\Big[\overline{\rho_{2n-1}}u(2n-1)-\alpha_{2n-1}u(2n)\Big]u(4m_i-2n-2)\\
	&\qquad\quad+\Big[\alpha_{2n}(\overline{\alpha_{2n+1}}  u(2n+1)
	+\rho_{2n+1} u(2n+2))\Big]u(4m_i-2n-1)\\
	&\qquad\quad
	+z(u(2n)u(4m_i-2n-2)+u(2n+1)u(4m_i-2n-1))\Big|\\
	&\qquad=\frac1{|\overline{\rho_{2n}}|}\Big|\Big[(\overline{\rho_{2n}}\alpha_{2n-1}+\rho_{4m_i-2n-2}\overline{\alpha_{4m_i-2n-1}})u(4m_i-2n-1)-(\overline{\rho_{2n}}\overline{\rho_{2n-1}}-\rho_{4m_i-2n-2}\rho_{4m_i-2n-1})u(4m_i-2n)\Big]u(2n)\\
	&\qquad\quad-\Big[zu(4m_i-2n-2)-\overline{\alpha_{4m_i-2n-2}}(\overline{\rho_{4m_i-2n-3}}u(4m_i-2n-3)  -\alpha_{4m_i-2n-3} u(4m_i-2n-2))\Big]u(2n)\\
	&\qquad\quad-\Big[(\rho_{2n}\overline{\alpha_{2n+1}}+\overline{\rho_{4m_i-2n-2}}\alpha_{4m_i-2n-3})u(4m_i-2n-2)+(\rho_{2n+1}\rho_{2n}-\overline{\rho_{4m_i-2n-2}}\overline{\rho_{4m_i-2n-3}})u(4m_i-2n-3)\Big]u(2n+1)\\
	&\qquad\quad-\Big[zu(4m_i-2n-1)+\alpha_{4m_i-2n-2}(\overline{\alpha_{4m_i-2n-1}}  u(4m_i-2n-1)+\rho_{4m_i-2n-1} u(4m_i-2n))\Big]u(2n+1)\\
	&\qquad\quad
	-\overline{\alpha_{2n}}\Big[\overline{\rho_{2n-1}}u(2n-1)-\alpha_{2n-1}u(2n)\Big]u(4m_i-2n-2)\\
	&\qquad\quad+\Big[\alpha_{2n}(\overline{\alpha_{2n+1}}  u(2n+1)
	+\rho_{2n+1} u(2n+2))\Big]u(4m_i-2n-1)\\
	&\qquad\quad
	+z(u(2n)u(4m_i-2n-2)+u(2n+1)u(4m_i-2n-1))\Big|\\
	&\qquad=\frac1{|\overline{\rho_{2n}}|}\Big|\Big[(\overline{\rho_{2n}}\alpha_{2n-1}+\rho_{4m_i-2n-2}\overline{\alpha_{4m_i-2n-1}})u(4m_i-2n-1)-(\overline{\rho_{2n}}\overline{\rho_{2n-1}}-\rho_{4m_i-2n-2}\rho_{4m_i-2n-1})u(4m_i-2n)\Big]u(2n)\\
	&\qquad\quad-\Big[(\rho_{2n}\overline{\alpha_{2n+1}}+\overline{\rho_{4m_i-2n-2}}\alpha_{4m_i-2n-3})u(4m_i-2n-2)+(\rho_{2n+1}\rho_{2n}-\overline{\rho_{4m_i-2n-2}}\overline{\rho_{4m_i-2n-3}})u(4m_i-2n-3)\Big]u(2n+1)\\	&\qquad\quad+\Big[\overline{\alpha_{4m_i-2n-2}}(\overline{\rho_{4m_i-2n-3}}u(4m_i-2n-3)  -\alpha_{4m_i-2n-3} u(4m_i-2n-2))\Big]u(2n)\\
	&\qquad\quad-\Big[\alpha_{4m_i-2n-2}(\overline{\alpha_{4m_i-2n-1}}  u(4m_i-2n-1)+\rho_{4m_i-2n-1} u(4m_i-2n))\Big]u(2n+1)\\
	&\qquad\quad
	-\overline{\alpha_{2n}}\Big[\overline{\rho_{2n-1}}u(2n-1)-\alpha_{2n-1}u(2n)\Big]u(4m_i-2n-2)\\
	&\qquad\quad+\Big[\alpha_{2n}(\overline{\alpha_{2n+1}}  u(2n+1)
	+\rho_{2n+1} u(2n+2))\Big]u(4m_i-2n-1)\Big|.
	\end{align*}\normalsize
Now applying \eqref{eq:LMeig1} and \eqref{eq:LMeig2} transforms the final expession into\footnotesize
\begin{align*}
	&\qquad=\frac1{|\overline{\rho_{2n}}|}\Big|\Big[(\overline{\rho_{2n}}\alpha_{2n-1}+\rho_{4m_i-2n-2}\overline{\alpha_{4m_i-2n-1}})u(4m_i-2n-1)-(\overline{\rho_{2n}}\overline{\rho_{2n-1}}-\rho_{4m_i-2n-2}\rho_{4m_i-2n-1})u(4m_i-2n)\Big]u(2n)\\
	&\qquad\quad-\Big[(\rho_{2n}\overline{\alpha_{2n+1}}+\overline{\rho_{4m_i-2n-2}}\alpha_{4m_i-2n-3})u(4m_i-2n-2)+(\rho_{2n+1}\rho_{2n}-\overline{\rho_{4m_i-2n-2}}\overline{\rho_{4m_i-2n-3}})u(4m_i-2n-3)\Big]u(2n+1)\\	&\qquad\quad+\Big[\overline{\alpha_{4m_i-2n-2}}z(\alpha_{4m_i-2n-2}u(4m_i-2n-2)  +\rho_{4m_i-2n-2} u(4m_i-2n-1))\Big]u(2n)\\
		&\qquad\quad-\Big[\alpha_{4m_i-2n-2}z(\overline{\rho_{4m_i-2n-2}}  u(4m_i-2n-2)-\overline{\alpha_{4m_i-2n-2}} u(4m_i-2n-1))\Big]u(2n+1)\\
	&\qquad\quad
	-\overline{\alpha_{2n}}z\Big[\alpha_{2n}u(2n)+\rho_{2n}u(2n+1)\Big]u(4m_i-2n-2)\\
	&\qquad\quad+\Big[\alpha_{2n}z(\overline{\rho_{2n}}  u(2n)
	-\overline{\alpha_{2n}} u(2n+1))\Big]u(4m_i-2n-1)\Big|.
\end{align*}\normalsize
This can in turn be rearranged into
\begin{align*}
	&\qquad=\frac1{|\overline{\rho_{2n}}|}\Big|(\overline{\rho_{2n}}\alpha_{2n-1}+\rho_{4m_i-2n-2}\overline{\alpha_{4m_i-2n-1}})u(4m_i-2n-1)u(2n)\\
	&\qquad\quad-(\overline{\rho_{2n}}\overline{\rho_{2n-1}}-\rho_{4m_i-2n-2}\rho_{4m_i-2n-1})u(4m_i-2n)u(2n)\\
	&\qquad\quad-(\rho_{2n}\overline{\alpha_{2n+1}}+\overline{\rho_{4m_i-2n-2}}\alpha_{4m_i-2n-3})u(4m_i-2n-2)u(2n+1)\\
	&\qquad\quad+(\rho_{2n+1}\rho_{2n}-\overline{\rho_{4m_i-2n-2}}\overline{\rho_{4m_i-2n-3}})u(4m_i-2n-3)u(2n+1)\\	
	&\qquad\quad+z\Big[\overline{\alpha_{4m_i-2n-2}}\alpha_{4m_i-2n-2}-\overline{\alpha_{2n}}\alpha_{2n}\Big]u(4m_i-2n-2)u(2n)\\
	&\qquad\quad+z\Big[\overline{\alpha_{4m_i-2n-2}}\rho_{4m_i-2n-2}+\overline{\rho_{2n}}\alpha_{2n}\Big]u(4m_i-2n-1)u(2n)\\
	&\qquad\quad-z\Big[\alpha_{4m_i-2n-2}\overline{\rho_{4m_i-2n-2}}+\overline{\alpha_{2n}}\rho_{2n}\Big]u(4m_i-2n-2)u(2n+1)\\
	&\qquad\quad+z\Big[\alpha_{4m_i-2n-2}\overline{\alpha_{4m_i-2n-2}}-\overline{\alpha_{2n}}\alpha_{2n} \Big]u(4m_i-2n-1))u(2n+1)\Big|,
\end{align*}
which is \eqref{eq:longWronskiCalc}.
\end{appendix}

\bibliographystyle{abbrvArXiv}

\bibliography{jito-sim-bib}

\end{document}